\documentclass[leqno]{article}
\pdfoutput=1

\usepackage{amssymb,amsfonts,amsthm,amsmath} 
\usepackage{a4wide}
\usepackage{graphicx}
\usepackage[usenames,dvipsnames,svgnames,table]{xcolor}
\usepackage[colorlinks=true, pdfstartview=FitV,linkcolor=ForestGreen,citecolor=ForestGreen, urlcolor=blue]{hyperref}
\usepackage[utf8]{inputenc}
\usepackage[T1]{fontenc}
\usepackage[margin=1in]{geometry} 
\usepackage{esint} 
\usepackage{enumitem}

\newcommand{\R}{\mathbb{R}}

\newcommand{\eps}{\varepsilon}

\newtheorem{theo}{\bf Theorem}[section]
\newtheorem{lem}[theo]{\bf Lemma}

\newtheorem{cor}[theo]{\bf Corollary}
\newtheorem{defi}[theo]{\bf Definition}
\theoremstyle{remark}
\newtheorem{rem}[theo]{Remark}

\numberwithin{equation}{section}

\begin{document}

\title{Coercive Hamilton-Jacobi equations in domains:\\
the twin blow-ups method}
\author{Nicolas Forcadel, Cyril Imbert et Régis Monneau}
\date{\today}

\maketitle

\begin{abstract}
  In this note, we consider an evolution coercive Hamilton-Jacobi equation posed in a domain and supplemented with a boundary condition.
  We are interested in proving a comparison principle in the case where the time and the (normal) gradient variables are strongly coupled at the boundary.
  We elaborate on a method introduced by P.-L. Lions and P. Souganidis (\textit{Atti Accad. Naz. Lincei}, 2017) to extend their comparison principle to more general boundary conditions
  and to Hamiltonians that are not globally Lipschitz continuous in the time variable.
  Their argument relies on a {single} blow-up procedure after rescaling the semi-solutions to be compared. We refer to our technique as the \emph{twin blow-ups method}
  since two blow-ups are performed simultaneously,  one for each  variable of the doubling variable method.
  The Lipschitz regularity of the regularized subsolution provides a key Lipschitz inequality satisfied by the two blow-up limits, that are a priori allowed to be infinite.
  For expository reasons, the result is presented here in the framework of space dimension one and the general case is treated in a companion paper.
\end{abstract}


\bigskip

\section{Introduction: a comparison principle}

{Given $T>0$, we} consider viscosity solutions of a standard evolutive Hamilton-Jacobi equation posed {in the geometric setting of  a domain} $\Omega:=(0,+\infty)$,
\begin{equation}\label{eq:hj}
u_t + H(X,u_x)=0 \quad \text{in}\quad (0,T)\times \Omega
\end{equation}
where $X:=(t,x)$, supplemented with the boundary condition,
\[
  u_t+F(X,u_x)=0  \quad \text{in}\quad (0,T)\times \partial \Omega
\]
and the initial condition,
\[
u(0,\cdot)=u_0 \quad \text{in}\quad \left\{0\right\}\times \overline \Omega.
\]
Since the boundary condition can be lost when characteristics reach $\partial \Omega$, it has to be imposed in a weak sense. 
When working with viscosity solutions, a classical way to handle this discrepancy is to impose that
either the boundary condition {or} the equation is satisfied (in the viscosity sense) at the boundary,
\begin{equation}\label{eq:bc-w}
\left\{\begin{array}{lll}
u_t +\min\left\{F,H\right\}(X,u_x) \le 0 &\quad \text{in}\quad (0,T)\times \partial \Omega&\quad \text{(for subsolutions),}\\
u_t +\max\left\{F,H\right\}(X,u_x) \ge 0 &\quad \text{in}\quad (0,T)\times \partial \Omega &\quad \text{(for supersolutions).}
\end{array}\right.
\end{equation}

Comparison principles are strong uniqueness results for Hamilton-Jacobi equations. In the case of the previous initial boundary value problem, it is known (see \cite{L,B,B2,LS2,BC})
that it is difficult to treat the case when tangential coordinates, such as the time variable $t$, and the normal derivative $u_x$ of the solution, are strongly coupled {reaching} the boundary $(0,T)\times \partial \Omega$.
It is standard to make the (strong) assumption of uniform continuity in time $t$, uniformly in the gradient $u_x$. Such an assumption is not satisfied by the following simple example,
\begin{equation} \label{eq:hj-ex}
  \begin{cases}
u_t + a(t,x)|u_x|=0 & \quad \text{in}\quad (0,T)\times  \Omega,\\
u_t - b(t,x)u_x=0 &\quad \text{in}\quad (0,T)\times \partial \Omega
\end{cases}
\end{equation}
when $a,b\ge 1$ are bounded Lipschitz continuous functions {(here with $b(t,x)=b(t,0)$)}.

In this note, we choose structural assumptions on $H$ and $F$ that encompass a large variety of examples (including \eqref{eq:hj-ex}) but
we do not seek for generality to avoid  technicalities in proofs as much as possible. 
We assume that $H$ and $F$ are continuous in $X$, Lipschitz continuous and (semi-) coercive in $p$, with a time dependance allowing strong coupling with the gradient variable. 
Precisely, we assume that there exists a  constant $C>0$ such that,
\begin{equation}\label{assum:H}
\begin{cases}
 H \text{ is continuous and } |H(X,0)| \le C \text{ and } |H(X,p)-H(X,q)| \le C|p-q| \\
 H(X,p) \to + \infty \text{ as } |p|\to \infty \text{ uniformly in } X \\
  |H(s,x,p)-H(t,x,p)| \le C|t-s|\left(1+\max(0,H)(t,x,p) \right)
\end{cases}
\end{equation}
 \begin{equation}\label{assum:F}
 \begin{cases}
 F  \text{ is continuous and } |F(X,0)| \le C \text{ and } |F(X,p)-F(X,q)| \le C|p-q|\\
F(X,p) \to +\infty \text{ as } p \to -\infty \text{ uniformly in } X \\
 |F(s,x,p)-F(t,x,p)| \le C|t-s|\left(1+\max(0,F,H)(t,x,p)\right).
\end{cases}
\end{equation}
We make artificially appear the dependence of $F$ on $x\in \overline \Omega$  in order to unify the presentation of assumptions for $H$ and $F$.

By regularizing  the subsolution by tangential sup-convolution, and \emph{without regularizing the supersolution}, we are able to prove the following
new comparison principle. 
\begin{theo}[A comparison principle with strong time coupling]\label{th:comparison}
Let $T>0$, assume that \eqref{assum:H} and \eqref{assum:F} hold true and that $u_0$ is bounded and Lipschitz continuous.
Let $u:[0,T)\times \overline \Omega \to \R$ (resp. $v$) be a bounded upper semi-continuous viscosity subsolution (resp. bounded lower semi-continuous viscosity supersolution)
of \eqref{eq:hj}-\eqref{eq:bc-w}. If
\(u(0,\cdot)\le u_0 \le v(0,\cdot)\) in $\overline \Omega$, then 
\(u\le v\) in \([0,T)\times \overline \Omega\).
\end{theo}
\begin{rem}\label{rem:: n41}
Here,  we assume for simplicity that the initial data $u_0$ is Lipschitz continuous.
With some additional (classical) technicalities, it is possible to deal with uniformly continuous $u_0$'s
and to relax the boundedness assumption on $u,v,u_0$ by imposing that they grow at most linearly.
\end{rem}
\begin{rem}\label{rem::b42}
Lipschitz continuity of $F$ with respect to $p$ is only used to get barriers in the proof of Theorem~\ref{th:comparison}.
Moreover, the monotonicity of $F$ with respect to $p$ is not used in the proof of the comparison principle. However it is required in order to ensure that classical solutions are viscosity solutions. This important and natural property is used when constructing viscosity solutions by Perron's method.
\end{rem}
\begin{rem}
Notice that, given (\ref{assum:H}), we can always define the state constraint boundary function
$$H^-(t,x,p):=\sup_{q\le p} H(t,x,q)\quad \text{for}\quad (t,x)\in [0,T]\times \partial\Omega\quad \text{and}\quad p\in \R$$
and it satisfies (\ref{assum:F}). {Up to our knowledge, the comparison principle was also an open problem for  $F=H^-$ in the generality of this note.}
\end{rem}
\begin{rem}
{Notice that in Theorem \ref{th:comparison}, when $F$ is nonincreasing in $p$, condition (\ref{assum:F}) iv) can be replaced by the weak continuity of the subsolution $u$ on the boundary $(0,T)\times \partial \Omega$, using \cite[Proposition~3.12]{FIM}) and replacing $F$ by $F_1:=\max(F,H^-)$.}
\end{rem}

\paragraph{Comparison with known results.}

J. Guerand \cite{zbMATH06731793} proved a comparison principle in our geometric setting in the case where  $H$ and $F$ {are} independent of $(t,x)$. 
She also proved a comparison principle for non-coercive Hamiltonians. 

P.-L. Lions and P. Souganidis \cite{LS2} introduced a new method to prove  comparison principle for junctions
with $N \ge 1$ branches (or half-spaces) between bounded uniformly continuous sub/supersolutions. They use
a blow-up argument that reduces the study to a 1D problem. The authors show the comparison principle in the case of Kirchoff-type boundary conditions
and non-convex Hamiltonians.
As far as $(t,x)$ dependence is concerned, these authors can handle Hamiltonians that are
Lipschitz in $t$, see  \cite[Assumption (4)]{LS2}.

 This result is generalized by G. Barles and E. Chasseigne \cite[Theorem 15.3.7, page 295]{BC} to the case of bounded semi-continuous sub/supersolutions under 
three different junction conditions. Even if they are presented for $N = 2$ branches, we present their results in our geometric setting: 
a junction reduced to a single branch $N = 1$ in 1D. 
The three cases are the following: $F$ is constant in \eqref{eq:bc-w}, the Neumann problem, and general nonincreasing continuous $F$.
 In the third case, the normal derivative is not coupled with the tangential coordinates in $F$ (see also the very end of  \cite[\S 13.2.2 and condition (GA-G-FLT) p. 247]{BC}).

As explained above, we improve these results in the case where the functions $H$ and $F$ imply a strong coupling of the time variable with the normal derivative of the solution. 
 
\paragraph{Organization of the note.}
In Section~\ref{s:boundary} we recall the definition of (limiting) semi-differentials and we state and prove a technical lemma relating some slopes (that we call critical)
at the boundary with semi-differentials (Lemma~\ref{lem:limit-diff}).
We deduce from this technical lemma a corollary about critical slopes of stationary semi-solutions of the boundary value problem (Corollary~\ref{cor:visco-ineq}). In Section~\ref{s:barriers}, we state a {barrier result (Lemma~\ref{lem:barriers})} that helps us dealing with the initial time condition; we also state a {result (Lemma~ \ref{lem:regularization})} about regularized subsolutions. In Section~\ref{s:proof},
the comparison principle (Theorem~\ref{th:comparison}) is proved. 

\paragraph{Acknowledgements.} The authors thank G. Barles and E. Chasseigne for enlighting discussions during the preparation of this note.
 This research was funded, in whole or in part, by l'Agence Nationale de la Recherche (ANR), project ANR-22-CE40-0010. For the purpose of open access, the author has applied a CC-BY public copyright licence to any Author Accepted Manuscript (AAM) version arising from this submission.

\section{Boundary lemmas}
\label{s:boundary}

\begin{defi}[(Limiting) semi-differentials]
For a set $A\subset \overline \Omega=[0,+\infty)$ and a point $x_0\in A$,  we define the (first order) super/subdifferential at $x_0$ of a function $u$ in $A$ as
\[
D^\pm_A u(x_0)=\left\{p\in \R,\quad \text{such that}\quad 0\le \pm\left\{u(x_0)+p(x-x_0)+o(x-x_0)-u(x)\right\}\quad \text{in}\quad A\right\}\
\]
and the limiting (first order) super/subdiffential at the boundary $x_0=0$ of $u$ in $\Omega$ as
\[
\bar D^\pm_\Omega u(0)=\left\{\text{$p\in \R$, there exists a sequence $p_n\in D^\pm_\Omega u(x_n)$ with $x_n\in \Omega$ and $(x_n,p_n)\to (0,p)$}\right\}.
\]
\end{defi}
\begin{rem}
Note that if $p\in  \bar D^+_\Omega u(0)$ and $u$ is a subsolution of $H(u_x)\le 0$ in $\Omega$, then $H(p)\le 0$.
\end{rem}
\begin{lem}[{Critical slopes and semi-differentials}]\label{lem:limit-diff}
Let $\Omega:=(0,+\infty)$ and  $u,v: \overline \Omega \to \R\cup \left\{-\infty,+\infty\right\}$ with $u$ upper semi-continuous and $v$ lower semicontinous such that $u(0)=0=v(0)$ and $u\le v$ in $\overline \Omega$.
The critical slopes defined by
\begin{equation}\label{eq::n12}
\overline p:=\limsup_{\Omega\ni x \to 0} \frac{u(x)}{x},\quad \underline p:=\liminf_{\Omega\ni x \to 0} \frac{v(x)}{x}
\end{equation}
satisfy the following (limiting) semi-differentials inclusions
\begin{equation}\label{eq::c1}
\R\cap \left[\underline p, \overline p\right]\subset \bar D^+_\Omega u(0) \cap \bar D^-_\Omega v(0)\quad \mbox{if}\quad  \overline p\ge \underline p
\end{equation}
\begin{equation}\label{eq::c2}
\R\cap \left[\overline p,\underline p\right] \subset  D^+_{\overline \Omega} u(0)\cap  D^-_{\overline \Omega} v(0) 
\quad \mbox{if}\quad \overline p \le \underline p
\end{equation}
{\begin{equation}\label{eq::c3}
\left\{\begin{array}{ll}
\overline p \in \bar D^+_\Omega u(0) &\quad \mbox{if}\quad \overline p \in \R \\
\underline p \in \bar D^-_\Omega v(0) &\quad \mbox{if}\quad \underline p \in \R \\
\end{array}\right.
\end{equation}}
\end{lem}

\begin{proof}
We first notice that (\ref{eq::c2}) is a straightforward consequence of the  definition of semi-differentials.

We now focus on the proof of 
\begin{equation}\label{eq::n13}
\R\cap \left[\underline p, \overline p\right]\subset \bar D^+_\Omega u(0)  \quad \mbox{in case}\quad \overline p > \underline p
\end{equation}
and will even show the following better result
\begin{equation}\label{eq::n13bis}
\R\cap \left[\underline q, \overline p\right]\subset \bar D^+_\Omega u(0)  \quad \mbox{in case}\quad \overline p >  \underline q := \liminf_{\Omega\ni x\to 0} \frac{u(x)}{x}
\end{equation}
where the last inequality is assumed because $u\le v$ implies $\underline q\le \underline p$. This claim is a variant of  \cite[Eq.~(18)]{LS2}, and its
proof is a variant of   \cite[Lemma~15.3.1]{BC}. We give the details for the sake of completness.

Notice that  $p\in (\underline q,\overline p)$ means
$$\limsup_{\Omega\ni x\to 0} \frac{u(x)}{x}=\overline p>  p >  \underline q = \liminf_{\Omega\ni x\to 0} \frac{u(x)}{x}$$
We deduce that for any $\varepsilon>0$, there exists $y_\varepsilon \in (0,\varepsilon)$ and $z_\varepsilon \in (0,y_\varepsilon)$ such that
$$\frac{u(z_\varepsilon)}{x_\varepsilon}\  > p >   \frac{u(y_\varepsilon)}{y_\varepsilon}$$
Hence the function $\zeta(x):=u(x)-px$ satisfies 
$\zeta(0)=0> \zeta(y_\varepsilon)$. Let $M:=\sup_{[0,y_\varepsilon]} \zeta \ge \zeta(z_\varepsilon)>0$. 
Hence at a point $x_\varepsilon \in (0,y_\varepsilon)$ of maximum of $\zeta$, we see that $p\in D^+_\Omega u(x_\varepsilon)$. In the limit $\varepsilon\to 0$, we recover $p\in \bar D^+_\Omega u(0)$. Then  (\ref{eq::n13bis}) follows from the fact that $\bar{D}_\Omega^+ u(0)$ is closed.
Then (\ref{eq::n13}) holds true for $u$. {The similar inclusion for $v$ implies (\ref{eq::c1}) in the special case where $\overline p>\underline p$.
On the other hand, notice that  (\ref{eq::c3}) implies (\ref{eq::c1}) in the case $\overline p=\underline p$.
Hence it remains to show (\ref{eq::c3})}.

We now explain why the following fact holds true,
{\begin{equation}\label{eq::n14}
\underline p \in \bar D^-_\Omega v(0) \quad \mbox{if}\quad \underline p \in \R
\end{equation}
This result is a property of the critical slope for  any lower semi-continuous functions.
Its proof follows exactly the lines of the proof of  \cite[Lemma~2.9]{IM1}.
A similar result holds for $u$ and this  proves (\ref{eq::c3}).}
\end{proof}

{\begin{defi}[Coercive and semi-coercive functions]
Consider a function $G:\R \to \R$. Then $G$ is \emph{coercice} if \(\lim_{|p|\to +\infty} G(p)=+\infty\), and semi-coercive if \(\lim_{p\to -\infty} G(p)=+\infty\).
\end{defi}}

As a consequence of Lemma \ref{lem:limit-diff}, we have
\begin{cor}[Boundary viscosity inequalities]\label{cor:visco-ineq}
Let $\Omega$ and $u$, $v$ be as in statement of Lemma \ref{lem:limit-diff}.
For $\gamma=\alpha,\beta$, consider continuous functions $H_\gamma,F_\gamma: \R\to \R$ with $H_\alpha$  coercive {and $F_\alpha$ semi-coercive}. Assume that we have the following viscosity inequalities for some $\eta>0$
\begin{equation}\label{eq::n7}
\left\{\begin{array}{rlrl}
H_\alpha(u_x) \le 0 &\quad  \mbox{on} & \quad \Omega &\cap \left\{|u|< +\infty\right\}\\
\min\left\{F_\alpha,H_\alpha\right\}(u_x) \le 0 &\quad  \mbox{on} & \quad \left\{0\right\} &\cap \left\{|u|< +\infty\right\}\\
\\
H_\beta(v_x) \ge  \eta &\quad  \mbox{on} & \quad \Omega &\cap \left\{|v|< +\infty\right\}\\
\max\left\{F_\beta,H_\beta\right\}(v_x) \ge \eta &\quad  \mbox{on} & \quad \left\{0\right\} &\cap \left\{|v|< +\infty\right\}\\
\end{array}\right.
\end{equation}
For $\underline p,\overline p$ defined in (\ref{eq::n12}), we set $a:=\min(\underline p,\overline p)$ and $b:=\max(\underline p,\overline p)$. Then  {$\overline p\in [a,b]\cap \R$} and there exists a real number $p\in [a,b]$ such that
$$\mbox{either}\quad H_\alpha(p)\le 0 < \eta \le  (H_\beta-H_\alpha)(p)\quad
\mbox{or}\quad \max(F_\alpha,H_\alpha)(p)\le 0 <  \eta \le (F_\beta-F_\alpha)(p)$$
\end{cor}

\begin{rem}\label{rem::b30}
Corollary \ref{cor:visco-ineq} can very easily be extended to the case of junctions where the Hamiltonians $H_\alpha$'s are coercive on each branch
{and the junction function is semi-coercive in the sense of \cite[Eq.~(2.2)]{IM1}}.
\end{rem}

\begin{rem}\label{rem::n11}
Under the assumptions of Corollary \ref{cor:visco-ineq}, we can show  that the subsolution $u$  is  globally Lipschitz continuous in $\overline \Omega$.
\end{rem}

\begin{proof}[Proof of Corollary \ref{cor:visco-ineq}]

{We  sketch  the proof that $\overline p\in\R$.}
Because $H_\alpha$ is coercive and $F_\alpha$ is semi-coercive, we know from \cite[Lemma~3.8]{FIM} that $u$ is weakly continuous at $x=0$, i.e.
\begin{equation}\label{eq::c4}
0=u(0)=\limsup_{\Omega\ni x\to 0^+} u(x).
\end{equation}
Then the proof  \cite[Lemma~ 2.10]{IM1} shows additionally that $\overline p > -\infty$. Now we  claim that we also have $\overline p< +\infty$. Indeed,   it can be seen by contradiction, leaving fall down above the graph of $u$ on $[0,y]$, some straight lines of  slopes $s=\frac{u(y)}{y}$ for large positive $s$ and using the equation satisfied by $u$. We conclude that $\overline p\in \R\cap [a,b]$.

Assume first that $\underline p \le \overline p$.
Here (\ref{eq::c1}) shows that $H_\alpha(p)\le 0 < \eta\le H_\beta(p)$ and then
$$\eta \le (H_\beta-H_\alpha)(p)\quad \mbox{for all}\quad p\in \left[\underline p, \overline p\right]\cap \R$$
which implies in particular the desired conclusion.

If  $\underline p > \overline p$, then we have  $[a,b]\subset (-\infty,+\infty]$ with $a<b$ and
\begin{equation}\label{eq::n9}
\left\{\begin{array}{rllll}
H_\alpha(a) \le & 0 &&&\quad \mbox{because}\quad a\in \R\\
&0 &< \eta &\le H_\beta(b)&\quad \mbox{if}\quad b\in \R\\
\min\left\{H_\alpha,F_\alpha\right\} \le &0 &< \eta &\le \max\left\{H_\beta,F_\beta\right\} &\quad \mbox{on}\quad   \left[a,b\right]\cap \R
\end{array}\right.
\end{equation}
where the last line follows {from (\ref{eq::c2}), and the first two lines follow from (\ref{eq::c3}).\\
\noindent \textsc{Intermediate claim.} Now we} assume by contradiction that there exists $\varepsilon>0$ (small enough) such that
\begin{equation}\label{eq::n10}
\left\{\begin{array}{ll}
\mbox{\bf i)}&H_\beta-H_\alpha < \eta-\varepsilon\quad \mbox{or}\quad \varepsilon< H_\alpha\\
\mbox{and}&\\
\mbox{\bf ii)}&F_\beta-F_\alpha < \eta \quad \mbox{or}\quad \varepsilon< \max(F_\alpha,H_\alpha)
\end{array}\right|\quad \mbox{for all}\quad p\in [a,b]\cap \R
\end{equation}
Recall that the coercivity of $H_\alpha$ means $H_\alpha(\pm \infty):=\liminf_{p\to \pm \infty} H_\alpha(p)=+\infty$.\\
\noindent {\bf Case 1: $H_\alpha(b)>\varepsilon$}\\
Here $b$ can be finite or equal to $+\infty$. We get 
$$H_\alpha(b)> \varepsilon> 0\ge H_\alpha(a)$$
Therefore by continuity, there exists $p\in (a,b)$ such that $H_\alpha(p)=\varepsilon$.
Hence in the last line of (\ref{eq::n9}), the first inequality implies that $F_\alpha(p)\le 0$.
Because  (\ref{eq::n10}) i) and ii) hold true for $p$, we get 
$$H_\beta(p)<\eta\quad \mbox{and}\quad F_\beta(p)<\eta$$
which leads to a contradiction in the last line of (\ref{eq::n9}), the second inequality.\\
\noindent {\bf Case 2: $H_\alpha(b)\le \varepsilon$}\\
Then $b\in \R$ and (\ref{eq::n10}) i) implies for $p=b$ that $H_\beta(b)< \eta$, which is in contradiction with the second line of (\ref{eq::n9}).\\
{\noindent \textsc{Conclusion.} We just proved that  (\ref{eq::n10}) does not hold true.} This implies that for all $\varepsilon>0$ small enough, there exists some $p_\varepsilon\in [a,b]\cap \R$ such that we have at $p_\varepsilon$
$$\mbox{\bf i)}\quad H_\alpha\le \varepsilon < \eta-\varepsilon \le H_\beta-H_\alpha\quad \mbox{or}\quad \mbox{\bf ii)}\quad \max(F_\alpha,H_\alpha)\le \varepsilon < \eta \le F_\beta-F_\alpha$$
Because $H_\alpha$ is coercive, we see in both cases i) or ii), that we can always extract a subsequence as $\varepsilon\to 0$ such that 
$p_\varepsilon \to p\in [a,b]\cap \R$. As a consequence, we get that $p$ satisfies i) or ii) for $\varepsilon=0$, {which is the desired conclusion}.
\end{proof}

\section{Barriers and tangential regularization}
\label{s:barriers}

In the proof of the comparison principle, two standard results about the construction of barriers and regularization of subsolutions by sup-convolution are needed.
\begin{lem}[Barriers] \label{lem:barriers}
Assume \eqref{assum:H} and \eqref{assum:F} and that the initial data $u_0$ is bounded and Lipschitz continuous. 
Assume that $u$ (resp. $v$) is a bounded upper semi-continuous subsolution (resp. a  bounded lower semi-continuous supersolution) of \eqref{eq:hj}-\eqref{eq:bc-w}.
Then there exists some constant $\lambda>0$  such that the functions
\[u^\pm(t,x)=u_0(x)\pm \lambda t\]
satisfy the following barrier properties:
\begin{itemize}
  \item
if $u\le u_0$ in $\left\{0\right\}\times \overline \Omega$, then $u\le u^+$ in $[0,T)\times \overline \Omega$,
\item
if $v\ge u_0$ in $\left\{0\right\}\times \overline \Omega$, then $v  \ge  u^-$ in $[0,T)\times \overline \Omega$.
\end{itemize}
\end{lem}
The previous lemma is a direct consequence of the definition of viscosity solutions if $u_0$ is $C^1$. In the general case, it follows by a standard approximation procedure.

We now turn to the regularization of subsolutions with respect to tangential variables. Even if the proof is very standard, {we will give below a short sketch of it}.
\begin{lem}[Tangential regularization of subsolutions by sup-convolution] \label{lem:regularization}
Assume that $H$ satisfies (\ref{assum:H}). Let $u:[0,T)\times \overline \Omega \to \R$ be an upper semi-continuous subsolution of (\ref{eq:hj})
which satisfies
\[|u-u_0|\le C_T\quad \text{in}\quad [0,T)\times \overline \Omega.\]
We extend $u$ to $t=T$ by $u(T,x):= \limsup \{ u(s,y) : (s,y) \to (T,x), (s,y) \in [0,T)\times \overline \Omega \}$ and to $\R\times \overline \Omega$ by,
\[
  \begin{cases}
    u(t,x)=u(T,x) &\quad \text{for}\quad t\ge T,\\
    u(t,x)=u(0,x)&\quad \text{for}\quad t\le 0.
  \end{cases}
\]
Then for $\nu>0$, we consider the tangential sup-convolution
\[u^\nu(s,x):=\sup_{t\in \R} \left\{u(t,x)-\frac{|t-s|^2}{2\nu}\right\}= u(\bar t,x)-\frac{|\bar t-s|^2}{2\nu}\]
{and any such $\bar t$ (depending on $(s,x)\in \R\times \overline \Omega$) satisfies $\bar t\in [s-\theta^\nu,s+\theta^\nu]$}    with $\theta^\nu:=2\sqrt{\nu C_T}< T/2$.
\medskip

If $I^\nu$ denotes the time interval $(\theta^\nu,T-\theta^\nu)$, then
the function $u^\nu$ is Lipschitz continuous with respect to $t$ {in} $\R \times \bar \Omega$ and with respect to $x$ {in} $I^\nu \times \Omega$,
  \[ |\partial_s u^\nu|_{L^\infty(\R\times \overline \Omega)} \le \frac{\theta^\nu}{\nu} \quad \text{ and } \quad |\partial_x u^\nu|_{L^\infty(I^\nu\times \Omega)} \le L^\nu \]
  with $L^\nu:=\sup\left\{p\in \R,\quad \min_{X \in [0,T]\times \overline \Omega} H(X,p) \le \frac{\theta^\nu}{\nu} \right\}$.\\

  {Assume moreover that that $u$ is a subsolution at the boundary $(0,T)\times \partial\Omega$, i.e. satisfies the first line of (\ref{eq:bc-w}), for some $F$ that satisfies (\ref{assum:F}).
    Then $u^\nu$ is Lipschitz continuous in space and time on $I^\nu\times \overline \Omega$ of Lipschitz constant $L_\nu:=\max(\frac{\theta^\nu}{\nu},L^\nu)$.}
\end{lem}

\begin{proof}[Sketch of the proof]
It is easy to check that $u-u^\nu\le 2C_T$ which gives the bound on $\theta^\nu\ge|\bar t-s|$. Moreover the time derivative  of $u^\nu$ is like $\frac{\bar t-s}{\nu}$ which gives the bound on $\partial_s u^\nu$. The PDE inequality satisfied by $u^\nu$ gives naturally the bound on $\partial_x u^\nu$. Finally, when $F$ satisfies (\ref{assum:F}), we see using \cite[Lemma~3.8]{FIM}) that $u$ (and then $u^\nu$) is weakly continuous on $I^\nu\times \partial \Omega$, which implies the Lipschitz continuity of $u^\nu$ in $I^\nu\times \overline \Omega$.
This ends the sketch of the proof.
\end{proof}

\section{Proof of the comparison principle}
\label{s:proof}

Before proving our comparison principle, we describe the main steps. 

We first use the  doubling variable technique with respect to time with a parameter $\nu>0$.
This procedure can be interpreted as a sup-convolution in time of the subsolution $u$ only (\textbf{Step 1}).

We then focus on the case where the supremum of 
\[u(t,x)-v(s,x)-\text{correction/penalization} \]
is reached at some $(\bar t,\bar s,\bar x)$ with $\bar x$ on the boundary of $\Omega$. (\textbf{Step 2}). 

At this stage, the sup-convolution $u^\nu$ of the subsolution $u$ is  $L_\nu$- Lipschitz continuous in space and time (thanks to the coercivity of the Hamiltonian),
but the supersolution is only lower semi-continuous (\textbf{Step 3}).

We then consider twin blow-ups (\textbf{Step 4}): one at $(\bar t,\bar x)$ for $u$, and one at $(\bar s,\bar x)$ for $v$ (up to some correction terms on $v$).
After blow-up, we get half-relaxed limits $U^0$, $V^0$ that are globally defined on $\R_t\times \overline \Omega$ and  satisfy $U^0(0,0)=0=V^0(0,0)$. We show the following key {Lipschitz} inequality,
\begin{equation}\label{eq::b31}
U^0(t,x)-V^0(s,y)\le L_\nu |x-y|+b(t-s)\quad \text{with}\quad b:=\frac{\bar t-\bar s}{\nu}
\end{equation}
where the Lipschitz constant $L_\nu$ is inherited from $u^\nu$.

Then, considering  $\overline u$, the supremum in time of the map $t\mapsto U^0(t,x)-bt$ and $\underline v$, the infimum in time  of the map $s\mapsto  V^0(s,y)-bs$ (\textbf{Step 5}), we  see that $\overline u^*$ and $\underline v_*$ are  subsolution and supersolution of a 1D problem with moreover the key {Lipschitz}  inequality\footnote{Notice that without the $L_\nu$-Lipschitz inequality (\ref{eq::b31}), we would only get $\overline u \le \underline v$ and $\overline u(0)=0=\underline v(0)$, which is not sufficient to conclude.}
\[\overline u^*(x) - \underline v_*(y)\le L_\nu|x-y|\quad \text{with}\quad \overline u^*(0)=0=\underline v_*(0).\]
This procedure reduces the study to a 1D problem that is solved using the boundary viscosity inequalities from Corollary \ref{cor:visco-ineq} (\textbf{Step 6}).
\bigskip

\begin{proof}[Proof of Theorem~\ref{th:comparison}]
The proof is split into several steps. The first two steps are standard and new ideas appear in the next steps.
\medskip

\noindent {\bf Step 1: approximate supremum.}\\
Let $\eta >0$. This parameter {will be small enough but} will not vary until we {prove} that the following quantity is non-positive,
\[ M = \sup_{t \in [0,T), x \in \overline \Omega} \left\{ u(t,x) - v(t,x) - \frac{\eta}{T-t} \right\}. \]
It turns out that 
\begin{equation}
  \label{eq:M}
M = \lim_{\nu\to 0} \left\{\lim_{\alpha\to 0} M_{\nu,\alpha}\right\}.
\end{equation}
with
\[M_{\nu,\alpha}:=\sup_{t,s\in [0,T),\ x\in \overline \Omega} \Psi_{\nu,\alpha}(t,s,x){=\Psi_{\nu,\alpha}(\bar t,\bar s,\bar x)}\]
and {(with a careful choice of the penalization term $\frac{\eta}{T-s}$ instead of $\frac{\eta}{T-t}$)}
\[\Psi_{\nu,\alpha}(t,s,x):=u(t,x)-v(s,x)-\frac{\eta}{T-s}-\alpha g(x) - \frac{|t-s|^2}{2\nu}\quad \text{with}\quad g(x):=\frac{x^2}{2}.\]
{Moreover all} maximisers $(\bar t,\bar s,\bar x)$ in the definition of $M_{\nu,\alpha}$  satisfy 
\begin{equation}\label{eq::n35}
\lim_{\nu\to 0} \left\{\lim_{\alpha\to 0} \frac{|\bar t-\bar s|^2}{\nu}\right\}=0, \quad \lim_{\nu\to 0} \left\{\lim_{\alpha\to 0} \alpha g(\bar x)\right\}=0,\quad  \limsup_{\nu\to 0} \left\{\limsup_{\alpha\to 0}  \frac{\eta}{T-\bar s}\right\}\le 2C_T.
\end{equation}

\noindent {\textbf{ Step 2: Reduction to the case where the supremum is reached at the boundary.}}\\
Using the doubling variable technique with respect to $x$ for $u$ and $v$ and considering $u(t,x)$, $v(s,y)$ with a further penalization term of the form $\frac{|x-y|^2}{2\delta}$,
we can rely on barrier estimates close to $t=0$ to get estimates on maximum points $(\bar t_\delta,\bar x_\delta,\bar s_\delta,\bar y_\delta)\to (\bar t,\bar x,\bar s,\bar x)$ as $\delta\to 0$ {(for some subsequence $\delta$ and some suitable limit $(\bar t,\bar s,\bar x)$).}
We deduce that in the limit, the following fact holds true,
\begin{equation}\label{eq::n25}
\bar t,\bar s \in [\tau_\eta,T-\tau_\eta],\quad \bar x\in [0,\rho_\alpha],\quad \left(\frac{\bar t -\bar s}{\nu}, \bar p\right)\in \bar D^{1,+}_{t,x}  u(\bar t,\bar x),\quad  |\bar p|\le L_\nu
\end{equation}
where $L_\nu$ is the Lipschitz constant of $u^\nu$ {in}  Lemma~\ref{lem:regularization}.
Moreover, it is possible  to choose $\tau_\eta$  depending on $\eta$ only and $\rho_\alpha$  depending on $\alpha$ only.\footnote{Up to increase $\lambda$ and $C_T$ in the barrier Lemma \ref{lem:barriers},  and to decrease $\eta$ in the time penalisation term, we can assume that $\lambda T=C_T$ and it is possible to show in this case that we can choose $\tau_\eta:=\frac{\eta}{4C_T}$ and $\rho_\alpha=\sqrt{\frac{6C_T}{\alpha}}$.}

If $\bar x>0$, then we are in the classical case where we can conclude by writing viscosity inequalities for the sequence $(\bar t_\delta, \bar x_\delta, \bar s_\delta, \bar y_\delta)$  and
by combining them in the classical way. {Classical details for this step are given in the companion paper.} We are thus reduced to deal with the case where $\bar x=0$.

\medskip

\noindent {\textbf{Step 3: the  key Lipschitz estimate.}}\\
{Following Lemma~\ref{lem:regularization}, we extend $u$ and consider}
\[
  \begin{cases}
    \displaystyle U^\nu(s,x)=\sup_{t\in \R} \left\{u(t,x)-\frac{|t-s|^2}{2\nu}\right\}\\
    \displaystyle V(s,x)=v(s,x)+\frac{\eta}{T-s}+\alpha g(x)
  \end{cases}
\]
{and  there exists some (possibly non unique)} $\bar t_s \in [s-\theta^\nu,s+\theta^\nu]$ such that  $U^\nu (s,x)=u(\bar t_s,x)-\frac{|\bar t_s-s|^2}{2\nu}$.
If $s\in (\theta^\nu,T-\theta^\nu)$, then we see that $\bar t_s\in (0,T)$ and then we also have
\[U^\nu(s,x):=\sup_{t\in [0,T)} \left\{u(t,x)-\frac{|t-s|^2}{2\nu}\right\}.\]
  {In particular for $(s,x)=(\bar s,\bar x)$, we can choose $\bar t_{\bar s} = \bar t$ where $\bar t_{\bar s}$ is given by Lemma~\ref{lem:regularization} and $(\bar t,\bar s,\bar x)$ appear in Steps 1 and 2.}
Now we choose  $\nu>0$ small enough such that $\theta^\nu< \tau_\eta$ and we set $I^\nu:=(\theta^\nu,T-\theta^\nu).$
Moreover we have
$$u(t,x)-\frac{|t-s|^2}{2\nu} -V(s,x)=\Psi_{\nu,\alpha}(t,s,x)\le \Psi_{\nu,\alpha}(\bar t,\bar s,\bar x)=U^\nu(\bar s,\bar x)-V(\bar s,\bar x)$$
and then taking the supremum in $t\in\R$, we get for all $s\in I^\nu$, $x,y\in \overline \Omega$
{\begin{equation}\label{eq::n26}
\begin{cases}
U^\nu(s,y)-V(s,y) \le U^\nu(\bar s,\bar x)-V(\bar s,\bar x)\\
U^\nu(s,x)-U^\nu(s,y) \le L_\nu|x-y| 
\end{cases}
\end{equation}}
where the second line follows from the fact that $U^\nu$ is $L_\nu$-Lipschitz continuous (see  Lemma \ref{lem:regularization}).
Notice also that we have
\[
\left\{\begin{array}{rll}
\partial_t u + H(t,x,\partial_x  u) \le 0 &\quad \text{in}&\quad (0,T)\times \Omega\\
\partial_t u + \min(F,H)(t,x,\partial_x  u) \le 0 &\quad \text{in}&\quad (0,T)\times \partial\Omega\\[2ex]
-\frac{\eta}{(T-s)^2}+\partial_s V + H(s,x,\partial_x V-\alpha \partial_x g) \ge 0 &\quad \text{in}&\quad (0,T)\times \Omega\\
-\frac{\eta}{(T-s)^2}+\partial_s V + \max(F,H)(s,x,\partial_x V -\alpha \partial_x g) \ge 0 &\quad \text{in}&\quad (0,T)\times \partial\Omega.
\end{array}\right.
\]
\medskip

\noindent {\bf Step 4: the  twin blow-ups.}\\
We then consider the following  twin blow-ups with small parameter $\varepsilon>0$: one blow-up for  $u$ at the point $(\bar t,\bar x)$ and one blow-up for  $V$ at the point $(\bar s,\bar x)$,
\begin{equation}\label{eq::n27}
\begin{cases}
U^{\varepsilon}(\tau,\xi):=\varepsilon^{-1} \left\{  { u(\bar t+\varepsilon \tau, \bar x + \varepsilon \xi)- u (\bar t,\bar x)}\right\},&\quad U^{\varepsilon}(0,0)=0,\\
V^{\varepsilon}(\sigma,\xi):=\varepsilon^{-1} \left\{V(\bar s+\varepsilon \sigma, \bar x + \varepsilon \xi)-V(\bar s,\bar x)\right\},&\quad V^{\varepsilon}(0,0)=0.
\end{cases}
\end{equation}
Before passing to the limit $\eps \to 0,$ they satisfy
\begin{equation}\label{lequation}
\left\{\begin{array}{rll}
\partial_\tau U^\varepsilon + H(\bar t+\varepsilon \tau,\bar x+\varepsilon \xi,\partial_\xi U^\varepsilon) \le 0 &\quad \text{in}&\quad I^\varepsilon_{\bar t}\times \Omega\\
\partial_\tau U^\varepsilon + \min(F,H)(\bar t+\varepsilon \tau,\bar x+\varepsilon \xi,\partial_\xi U^\varepsilon) \le 0 &\quad \text{in}&\quad I^\varepsilon_{\bar t}\times \partial \Omega\\
\\
-\bar \eta^\varepsilon+\partial_\sigma V^\varepsilon + H(\bar s+ \varepsilon \sigma,\bar x+ \varepsilon \xi,\partial_\xi V^\varepsilon-\alpha \partial_x g(\bar x+\varepsilon \xi)) \ge 0 &\quad \text{in}&\quad I_{\bar s}^\varepsilon\times \Omega\\
-\bar \eta^\varepsilon+\partial_\sigma V^\varepsilon + \max(F,H)(\bar s+ \varepsilon \sigma,\bar x+ \varepsilon \xi,\partial_\xi V^\varepsilon -\alpha \partial_x g(\bar x+\varepsilon \xi)) \ge 0 &\quad \text{in}&\quad I_{\bar s}^\varepsilon\times \partial\Omega\\
\end{array}\right.
\end{equation}
with
\[
  \bar \eta^\varepsilon (\sigma):=\frac{\eta}{(T-(\bar s+ \varepsilon \sigma))^2}\quad \text{and}\quad {I_{\bar r}^\varepsilon:= \left(-\frac{\bar r}{\varepsilon},\frac{T-\bar r}{\varepsilon}\right)\quad \mbox{for}\quad \bar r=\bar t,\bar s.}
\]

Because $U^\nu$ is globally $L_\nu$-Lipschitz in space and time, it is also interesting to consider the auxiliary blow-up for $U^\nu$ at a point $(\bar s,\bar x)$,
\[
  U^{\nu,\varepsilon}(\sigma,\xi):=\varepsilon^{-1} \left\{ U^\nu(\bar s+\varepsilon \sigma, \bar x + \varepsilon \xi)-U^\nu(\bar s,\bar x)\right\},\quad U^{\nu,\varepsilon}(0,0)=0.
\]
This auxiliary function is indeed intimately related to $U^\varepsilon$ since we have,
\begin{equation}\label{eq::n29}
U^{\nu,\varepsilon}(\sigma,\xi)=\sup_{\tau\in \R} \left\{U^\varepsilon(\tau,\xi)-A^\varepsilon(\tau-\sigma)\right\},\quad A^\varepsilon(q)=bq+ \varepsilon\frac{|q|^2}{2\nu},\quad b:=\frac{\bar t-\bar s}{\nu}.
\end{equation}
Moreover from  (\ref{eq::n26}), we have for all $\sigma \in I_{\bar s}^\eps$ and $\xi,\zeta\in \overline \Omega$,
\[
  \begin{cases}
    U^{\nu,\varepsilon}(\sigma, \zeta) -V^\varepsilon (\sigma, \zeta) \le 0,\\
    {U^{\nu,\varepsilon}(\sigma,\xi)-U^{\nu,\varepsilon}(\sigma,\zeta) \le L_\nu|\xi-\zeta|}
  \end{cases}
\]
and then for all $\sigma \in I_{\bar s}^\eps$, $\tau \in \R$ and $\xi,\zeta\in \overline \Omega$,
\begin{equation}\label{eq::n30}
\left\{\begin{array}{l}
{U^{\nu,\varepsilon}(\sigma,\xi)-V^\varepsilon(\sigma,\zeta) \le L_\nu|\xi-\zeta|}\\
U^{\nu,\varepsilon}(\sigma,\xi)\ge U^\varepsilon(\tau,\xi)-A^\varepsilon(\tau-\sigma)\\
\end{array}\right.
\end{equation}
where the last line follows from (\ref{eq::n29}).

Thanks to the uniform Lipschitz estimate for $U^\nu$ and the fact that $U^{\nu,\eps} (0,0)=0$, we can extract a subsequence (still denoted by $\varepsilon\to 0$) and get,
\[
  U^{\nu,\varepsilon} \to U^{\nu,0}\quad \text{locally uniformly on compact sets of $\R\times \overline \Omega$},\quad U^{\nu,0}(0,0)=0.
\]
We then define (along the already extracted subsequence) the following half-relaxed limits
$$\left\{\begin{array}{ll}
\displaystyle U^0:=\limsup_{\varepsilon\to 0}{}^* U^\varepsilon,&\quad U^0(0,0)\ge 0,\\
\displaystyle V^0:=\liminf_{\varepsilon\to 0}{}_* V^\varepsilon,&\quad V^0(0,0)\le 0.
\end{array}\right.$$
Passing to the limit in (\ref{eq::n30}), we get with $A^0(q)=bq$
$$\left\{\begin{array}{l}
{U^{\nu,0}(\sigma,\xi)-V^0(\sigma,\zeta)\le L_\nu|\xi-\zeta|}\\
U^{\nu,0}(\sigma,\xi)\ge U^0(\tau,\xi)-b(\tau-\sigma)\\
\end{array}\right.$$
which shows in particular that
\begin{equation}\label{eq::n31}
{U^0(\tau,\xi)-V^0(\sigma,\zeta)\le L_\nu|\xi-\zeta|+b(\tau-\sigma)},\quad U^0(0,0)=0=V^0(0,0).
\end{equation}
Moreover, passing to the limit in \eqref{lequation} thanks to the discontinuous stability of viscosity solutions, we get
\begin{equation}\label{eq::n32}
\left\{\begin{array}{rlrl}
\partial_\tau U^0 + H(\bar t,\bar x,\partial_\xi U^0) \le 0 &\quad \text{in}&\quad (\R\times \Omega)&\cap \left\{|U^0|< +\infty\right\}\\
\partial_\tau U^0 + \min(F,H)(\bar t,\bar x,\partial_\xi U^0) \le 0 &\quad \text{in}&\quad (\R\times \partial \Omega)&\cap \left\{|U^0|< +\infty\right\}\\
\\
-\bar \eta+\partial_\sigma V^0 + H(\bar s,\bar x,\partial_\xi V^0) \ge 0 &\quad \text{in}&\quad (\R\times \Omega)&\cap \left\{|V^0|< +\infty\right\}\\
-\bar \eta+\partial_\sigma V^0 + \max(F,H)(\bar s,\bar x,\partial_\xi V^0 ) \ge 0 &\quad \text{in}&\quad  (\R\times \partial \Omega)&\cap \left\{|V^0|< +\infty\right\}
\end{array}\right.
\end{equation}
with $\bar \eta:=\frac{\eta}{(T-\bar s)^2}$. We used the fact that $\alpha \partial_x g(\bar x)=\alpha \bar x = 0$.
\medskip

\noindent \textbf{Step 5: the 1D problem.}\\
We now define the following functions as supremum/infimum in time of the functions defined in $\R\times \overline \Omega$,
\[
  \overline u(\xi):=\sup_{\tau\in \R} \left\{U^0(\tau,\xi)-b\tau\right\},\quad \underline v(\xi):=\inf_{\tau\in \R} \left\{V^0(\tau,\xi)-b\tau\right\}.
\]
From \eqref{eq::n31}, these functions satisfy
\[
  -\infty\le -L_\nu|\zeta-\xi|+\overline u(\xi) \le \underline v(\zeta)\le +\infty ,\quad 0\le \overline u(0) \le \underline{v}(0)\le 0
\]
and then $\overline u(0) =0= \underline{v}(0)$.
Because of this Lipschitz inequality, this is also the case for their semi-continuous envelopes, \textit{i.e.} we have
\begin{equation}\label{eq::n33}
-\infty\le -L_\nu|\zeta-\xi|+\overline u^*(\xi) \le \underline v_*(\zeta)\le +\infty ,\quad \overline u^*(0) =0= \underline{v}_*(0).
\end{equation}
From (\ref{eq::n32}), we get (again from  stability) that these functions satisfy in particular for $\bar X:=(\bar t,\bar x)$ and $\bar Y:=(\bar s,\bar x)$
\begin{equation}\label{eq::n34}
\left\{\begin{array}{rlrl}
b + H(\bar X,\partial_\xi \overline u^*) \le 0 &\quad \text{in}&\quad (\R\times \Omega)&\cap \left\{|\overline u^*|< +\infty\right\}\\
b + \min(F,H)(\bar X,\partial_\xi \overline u^*) \le 0 &\quad \text{in}&\quad (\R\times \partial \Omega)&\cap \left\{|\overline u^*|< +\infty\right\}\\
\\
-\bar \eta+b + H(\bar Y,\partial_\xi \underline v_*) \ge 0 &\quad \text{in}&\quad (\R\times \Omega)&\cap \left\{|\underline v_*|< +\infty\right\}\\
-\bar \eta+b + \max(F,H)(\bar Y,\partial_\xi \underline v_*) \ge 0 &\quad \text{in}&\quad  (\R\times \partial \Omega)&\cap \left\{|\underline v_*|< +\infty\right\}.
\end{array}\right.
\end{equation}

\noindent {\bf Step 6: getting a contradiction from structural assumptions.}\\
We now apply Corollary \ref{cor:visco-ineq}. In order to do so, we consider
\[
\overline p:=\limsup_{\Omega\ni x\to 0} \frac{\overline u^*(x)}{x},\quad \underline p:=\liminf_{\Omega\ni x\to 0}
\frac{\underline v_*(x)}{x},\quad a:=\min(\underline p, \overline p),\quad b:=\max(\underline p, \overline p)
\]
and we get that there exists $p\in [a,b]\cap \R\not= \emptyset$ such that either
\[
b + H(\bar X,p)\le 0 < \bar \eta \le H(\bar Y,p)-H(\bar X,p)
\]
or
\[
b + \max(F,H)(\bar X,p)\le 0 < \bar \eta \le F(\bar Y,p)-F(\bar X,p).
\]
{One} of these facts are true along a subsequence $\nu \to 0$. 
In the first case, we get from the assumption on the Hamiltonian $H$, see \eqref{assum:H} ii), that,
\begin{align*}
\bar \eta \le H(\bar Y,p)-H(\bar X,p) & \le C|\bar t-\bar s|\left\{1+\max(0,H(\bar X,p))\right\}\\
&\le C|\bar t-\bar s|\left\{1+\max(0,-b)\right\}\\
&\le  C\left\{\frac{|\bar t-\bar s|^2}{\nu} + |\bar t-\bar s|\right\} \to  0 \quad \text{as}\quad \nu\to 0
\end{align*}
where we have used $\displaystyle b=\frac{\bar t-\bar s}{\nu}$ in the third line, and (\ref{eq::n35}) in the last line. Contradiction because $\bar \eta\ge \eta/T^2>0$.\\
From the assumption on the function $F$, see \eqref{assum:F} ii), we  get a similar contradiction in the second case,
\begin{align*}
\bar \eta \le F(\bar Y,p)-F(\bar X,p) &\le C|\bar t-\bar s|\left\{1+\max(0,\max(F,H)(\bar X,p))\right\}\\
&\le C|\bar t-\bar s|\left\{1+\max(0,-b)\right\}\\
&\le C\left\{\frac{|\bar t-\bar s|^2}{\nu} + |\bar t-\bar s| \right\} \to 0 \quad \text{as} \quad \nu \to 0.
\end{align*}
We conclude that $M\le 0$. Recalling that
\[ M = \sup_{t \in [0,T), x \in \overline \Omega} \left\{ u(t,x) - v(t,x) - \frac{\eta}{T-t} \right\} \le 0, \]
it is enough to let $\eta \to 0$ to get $u \le v$ as desired. 
\end{proof}


\bibliographystyle{plain}
\bibliography{bib-comparison.bib}

\end{document}